\documentclass[entropy,article,submit,oneauthor,pdftex,12pt,a4paper]{mdpi} 
\lastpage{x}
\doinum{10.3390/------}
\pubvolume{xx}
\pubyear{2012}
\history{Received: xx / Accepted: xx / Published: xx}

\pdfoutput=1

\usepackage{amssymb,amsmath,amsthm,mathtools}
\usepackage[all]{xy}

\theoremstyle{plain}
\newtheorem{theorem}{Theorem}

\newtheorem{proposition}[theorem]{Proposition}
\theoremstyle{definition}
\newtheorem{definition}[theorem]{Definition}
\theoremstyle{remark}

\DeclareMathOperator{\Cov}{Cov}
\DeclareMathOperator{\E}{\mathbb E}

\DeclareMathOperator{\Vers}{Vers}
\newcommand{\KL}[2]{\operatorname{D}\left(#1\,\Vert#2\right)}
\newcommand{\absoluteval}[1]{\left|#1\right|}
\newcommand{\condexp}[2]{\E\left(#1 \left\vert #2 \right.\right)}
\newcommand{\covat}[3]{\Cov_{#1}\left(#2,#3\right)}

\newcommand{\densities}{\mathcal P_{\ge}}
\newcommand{\derivby}[1]{\frac{d}{d#1}}
\newcommand{\eBspace}[1]{B_{#1}}

\newcommand{\etransport}[2]{\prescript{\text{e}}{} {\mathbb U} _ {#1} ^ {#2}}
\newcommand{\euler}{\mathrm{e}}
\newcommand{\expectat}[2]{{\E}_{#1}\left[#2\right]}
\newcommand{\expectof}[1]{\E\left(#1\right)}
\newcommand{\expof}[1]{\exp\left(#1\right)}

\newcommand{\integrald}[3]{\int_{#1} {#2} \ {#3}}
\newcommand{\integrali}[4]{\int_{#1}^{#2} {#3} \ d{#4}}

\newcommand{\lnof}[1]{\ln\left(#1\right)}
\newcommand{\mBspace}[1]{\prescript{*}{}{B}_{#1}}
\newcommand{\maxexp}[1]{{\mathcal E}\left(#1\right)}

\newcommand{\mcoverat}[1]{\prescript{*}{}{\mathcal U}_{#1}}
\newcommand{\mtransport}[2]{\prescript{\text{m}}{} {\mathbb U} _ {#1} ^ {#2}}
\newcommand{\nonnegreals}{\reals_{\ge}}
\newcommand{\normat}[2]{\left\Vert#2\right\Vert_{#1}}

\newcommand{\pdensities}{\mathcal P_>}
\newcommand{\reals}{\mathbb R}
\newcommand{\scalarat}[3]{\left\langle#2,#3\right\rangle_{#1}}
\newcommand{\scalarof}[2]{\left\langle#1,#2\right\rangle}
\newcommand{\sdensities}{\mathcal P_{1}}
\newcommand{\sdomain}[1]{\mathcal S_{#1}}
\newcommand{\setof}[2]{\left\{#1 \colon #2 \right\}}
\newcommand{\set}[1]{\left\{#1\right\}}

\newcommand{\transport}[2]{\mathbb U_{#1}^{#2}}
\newcommand{\versof}[1]{\Vers\left(#1\right)}

\Title{Examples of Application of Nonparametric Information Geometry to Statistical Physics}

\Author{Giovanni Pistone}

\address[1]{%
de Castro Statistics Initiative, Collegio Carlo Alberto, Via Real Collegio 30, 10024 Moncalieri, Italy}

\corres{giovanni.pistone@carloalberto.org, Tel. +390116705033 Fax +390116705082}

\abstract{%
\nolinenumbers
We review a nonparametric version of Amari's Information Geometry in which the set of positive probability densities on a given sample space is endowed with an atlas of charts to form a differentiable manifold modeled on Orlicz Banach spaces. This nonparametric setting is used to discuss the setting of typical problems in Machine Learning and Statistical Physics, such as relaxed optimization, Kullback-Leibler divergence, Boltzmann entropy, Boltzmann equation.}


\begin{document}
\nolinenumbers

\section{Introduction}

Information Geometry was developed in the seminal monograph by Amari and Nagaoka \cite{amari|nagaoka:2000}, where previous---essentially metric---descriptions of probabilistic and statistics concepts are extended in the direction of differential geometry, including the fundamental treatment of differential connections. The differential geometry involved in their construction is finite dimensional and the formalism is based on coordinate systems. Following a suggestion by Phil Dawid in \cite{dawid:75,dawid:1977AS}, a particular nonparametric version of the Amari-Nagaoka theory was developed in a series of papers \cite{pistone|sempi:95,pistone|rogantin:99,gibilisco|pistone:98,gibilisco|isola:1999,cena:2002,cena|pistone:2007,imparato:thesis,pistone:2009EPJB,pistone:2009SL,pistone:2013GSI2013}, where the set $\pdensities$ of all strictly positive probability densities of a measure space is shown to be a Banach manifold (as defined in \cite{bourbaki:71variete,abraham|marsden|ratiu:1988,lang:1995}) modelled on an Orlicz Banach space (see e.g. \cite[Ch II]{musielak:1983}). 

Specifically, Gibbs densities $q = \euler^{u-K_p(u)} \cdot p$, $\expectat p u = 0$, are represented by the chart $s_p \colon q \mapsto u$. Because of the exponential form, the random variable $u$ is required to belong to an exponential Orlicz space, which is similar to an ordinary Lebesgue spaces, but lacks some important features of these spaces, such as reflexivity and separability. On the other side, the nonparametric setting emphasizes in a nice way the fact that statistical manifolds are actually affine manifolds with an Hessian structure, see \cite{shima:2007}. 

Such a formalism has been frequently criticised as unnecessarily involved to be of use in practical applications and also lacking really new results with respect to the Amari-Nagaoka theory. However, it should be observed that most applications in Statistical Physics, such as Boltzmann equation theory \cite{villani:2002review}, are intrinsically nonparametric. I like to quote here a line by Serge Lang in \cite[p. vi]{lang:1995}: ``One major function of finding proofs valid in the infinite dimensional case is to provide proofs which are especially natural and simple in the finite dimensional case.''

This paper is organized as follows. Sec.s \ref{sec:model-spaces} and \ref{exponentialmanifold} are a review of the basic material on statistical exponential manifolds with some emphasis on the functional analytic setting and on second order structures. Sec. \ref{sec:applications} contains a discussion of examples of application to the differential geometry of expected values, Kullback-Leibler divergence, Boltzmann entropy, Boltzmann equation. Sec. \ref{sec:conclusion} presents some among the topics that would require a further study, together with some lines of current research.

\section{Model spaces.}
\label{sec:model-spaces}

Given a $\sigma$-finite measure space $(\Omega, \mathcal F, \mu)$, we denote by $\pdensities$ the set of all densities which are positive $\mu$-a.s, by $\densities$ the set of all densities, by $\sdensities$ the set of measurable functions $f$ with $\int f\ d\mu = 1$. In the finite state space case $\sdensities$ is an affine subspace, $\densities$ is the simplex, $\pdensities$ its topological interior. We summarize below the basic notations and results. Missing proof are to be found e.g. in \cite{cena|pistone:2007} and in \cite[Ch II]{musielak:1983}.

If both $\phi$ and $\phi_*$ are monotone, continuous functions on $\nonnegreals$ onto itself such that $\phi^{-1}=\phi_*$, we call the pair

\begin{equation*}
  \Phi(x) = \int_0^{\absoluteval x} \phi(u) \ du, \quad \Phi_*(y) = \int_0^{\absoluteval y} \phi_*(v) \ dv,
\end{equation*}
a \emph{Young pair}. Each Young pair satisfies the Young inequality

\begin{equation}
  \label{eq:Young-inequality}
  \absoluteval{xy} \le \Phi(x) + \Phi_*(y)
\end{equation}
with equality if, and only if, $y = \phi(x)$. The relation in a Young pair is symmetric and either element is called a Young function. We will use the following Young pairs:

\begin{equation}\label{eq:Young-pairs}
  \begin{array}{l|c|c|c|c}
   & \phi_* & \phi=\phi_*^{-1} & \Phi_* & \Phi\\
\hline
\text{(a)} & \lnof{1+u} & \euler^v - 1 & (1+\absoluteval x)\lnof{1+\absoluteval x} - \absoluteval x & \euler^{\absoluteval y} - 1 - \absoluteval y \\
\text{(b)} & \sinh^{-1} u & \sinh v & \absoluteval x \sinh^{-1}\absoluteval x - \sqrt{1+x^2} + 1 & \cosh y - 1 \end{array}
\end{equation}
Let us derive a few elementary but crucial inequalities. If $x \ge 0$

\begin{equation}\label{eq:phistarsecond}
\Phi_*^{\text{(a)}}(x) = \integrali 0 x {\frac{x-u}{1+u}} {u} , \quad \Phi_*^{\text{(b)}}(x) = \integrali 0 x {\frac{x-u}{\sqrt{1+u^2}}} {u}, 
\end{equation}
hence, as $\sqrt{1+u^2} \le 1+ u \le \sqrt 2 \sqrt{1+u^2}$ if $u \ge 0$, for all real $x$ we have

\begin{equation}
  \label{eq:equivalent-Young-star}
  \Phi_*^{\text{(a)}}(x) \le \Phi_*^{\text{(b)}}(x) \le \sqrt 2 \Phi_*^{\text{(a)}}(x). 
\end{equation}
From \eqref{eq:phistarsecond} we have for $a > 1$

\begin{equation}\label{eq:delta2}
  \Phi_*^{\text{(a)}}(ax) = a^2 \integrali 0 {x} {\frac{x-v}{1+av}} {v} \le a^2 \Phi_*^{\text{(a)}}(x), \quad  \Phi_*^{\text{(b)}}(ax) = a^2 \integrali 0 {x} {\frac{x-v}{\sqrt{1+a^2v^2}}} {v} \le a^2 \Phi_*^{\text{(b)}}(x).
\end{equation}
In a similar way, from

\begin{equation*}
\Phi^{\text{(a)}}(y) = \integrali 0 y {(y-v)\euler^v} {v} , \quad \Phi^{\text{(b)}}(y) = \integrali 0 y {(y-v)\cosh v} {v}
\end{equation*}
and $\cosh v \le \euler^v \le 2\cosh v$ if $v \ge 0$ we have a relation similar to \eqref{eq:equivalent-Young-star}, that is for all $y$

\begin{equation}
  \label{eq:equivalent-Young}
  \Phi^{\text(b)}(y) \le \Phi^{\text(a)}(y) \le 2\Phi^{\text(b)}(y). 
\end{equation}
Property \eqref{eq:phistarsecond} does not hold in this case. Such an inequality is called $\Delta_2$-condition and has an important role in the theory of Orlizc spaces.

If $\Phi$ is any Young function, a real random variable $u$ belongs to the Orlicz space $L^{\Phi}(p)$ if $\expectat p {\Phi(\alpha v)} < +\infty$ for some $\alpha > 0$. A norm is obtained by defining the set $\setof{v}{\expectat p {\Phi(v)} \le 1}$ to be the closed unit ball. It follows that the open unit ball consists of those $u$'s such that $\alpha u$ is in the closed unit ball for some $\alpha > 1$. The corresponding norm $\Vert \cdot \Vert_{\Phi, p}$ is called Luxemburg norm and defines a Banach space, see e.g. \cite[Th 7.7]{musielak:1983}. From \eqref{eq:equivalent-Young} and \eqref{eq:equivalent-Young-star} follows that cases (a) and (b) in \eqref{eq:Young-pairs} define equal vector spaces with equivalent norms, , see \cite[Lemma 1]{cena|pistone:2007}, therefore we drop any mention of them. 

The Young function $\cosh -1$ has been chosen here because the condition $\expectat p {\Phi(\alpha v)} < +\infty$ is clearly equivalent to $\expectat p {\euler^{t v}} < +\infty$ for $ t \in [-\alpha,\alpha]$, that is the random variable $u$ has a Laplace transform around 0. The case of a moment generating function defined on all of the real line is special and define a notable subspace of the Orlicz space. The use of such space has been proposed by \cite{grasselli:2001}.

There are technical issues in working with Orlicz spaces such as $L^{(\cosh -1)}(p)$, in particular the regularity of its unit sphere $S_{\cosh-1} = \setof{u}{\normat {(\cosh-1),p} u = 1}$. In fact, while $\expectat p {\cosh u - 1} = 1$ implies $u \in S_{\cosh-1}$, the latter implies only $\expectat p {\cosh u - 1} \le 1$. Subspaces of $L^{\Phi}(p)$ where it cannot happen at the same time $\normat {(\cosh-1),p} u = 1$ and $\expectat p {\cosh u - 1} < 1$ are of special interest. In general, the sphere $S_{\cosh-1}$ is not smooth, see an example in \cite[Ex. 3]{pistone:2013GSI2013}.

If the functions $\Phi$ and $\Phi_*$ are Young pair, for each $u \in L^{\Phi}(p)$ and $v \in L^{\Phi_*}(p)$, such that $\normat {\Phi,p} u, \normat {\Phi_*,p} v \le 1$, we have from the Young inequality \eqref{eq:Young-inequality} $\absoluteval{\expectat p {uv}} \le 2$, hence

\begin{equation*}
  L^{\Phi_*}(p) \times L^{\Phi}(p) \ni (v,u) \mapsto \scalarat p u v = \expectat p {uv}
\end{equation*}
is a duality pairing, $\absoluteval{\scalarat p u v} \le 2 \normat {\Phi_*,p} u \normat {\Phi,p} v$ . It is a classical result that in our case \eqref{eq:Young-pairs} the space $L^{\Phi_*}(p)$ is separable and its dual space is $L^{\Phi}(p)$, the duality pairing being $(u,v) \mapsto \scalarat p uv$. This duality extends to a continuous chain of spaces:

\begin{equation*}
  L^{\Phi}(p) \to L^a(p) \to L^b(p) \to L^{\Phi_*}(p), \quad 1 < b \le 2, \quad \frac1a+\frac1b=1
\end{equation*}
where $\to$ denotes continuous injection.

\subsection{Cumulant generating functional}

Let $p \in \mathcal \pdensities$ be given. The following theorem has been proved in \cite[Ch 2]{cena:2002}, see also \cite{cena|pistone:2007}.
\begin{proposition}\label{prop:expisanalytic}
\begin{enumerate}
\item
For $a \geq 1$, $n = 0, 1, \dots$ and $u \in L^\Phi(p)$,

\begin{equation*}
 \lambda_{a,n}(u) \colon \left(w_1, \dots, w_n \right)  \mapsto  \dfrac{w_1}{a} \cdots \dfrac{w_n}{a}\ \euler^{\frac ua}
\end{equation*}
is a continuous, symmetric, $n$-multi-linear map from $L^\Phi(p)$ to $L^{a}\left(p\right)$.%
\item
$v \mapsto \sum_{n = 0}^\infty \frac{1}{n!} \left(\dfrac{v}{a}\right)^n$ is a power series from $L^{\Phi}(p)$ to $L^a(p)$ with radius of convergence $\geq 1$.
\item The superposition mapping $ v \mapsto \euler^{v/a}$ is an analytic function from the open unit ball of $L^\Phi(p)$ to $L^a(p)$.
\end{enumerate}
\end{proposition}

The previous theorem provides an improvement upon the original construction of \cite{pistone|sempi:95}.

\begin{definition}
Let $\Phi = \cosh -1$ and $\eBspace p = L^\Phi_0(p) = \setof{u \in L^\Phi_0(p)}{\expectat p u = 0}$, $p \in \pdensities$. The \emph{moment generating functional} is $M_p  \colon L^{\Phi}(p) \ni u \mapsto \expectat p {\euler^u} \in \reals_> \cup \set {+\infty}$. The \emph{cumulant generating functional} is $K_p  \colon \eBspace p \ni u \mapsto \log M_p(u) \in \reals_> \cup \set{+\infty}$.
\end{definition}

The moment generating functional is the partition functional (normalizing factor) of the Gibbs model, $(\euler^u/M_p(u)) \cdot p \in \pdensities$ if $u \in L^{\Phi}(p)$, $M_p(u) < +\infty$. The same model is written $\euler^{u-K_p(u)}\cdot p$ if moreover $\expectat p u = 0$.

\begin{proposition} \label{prop:cumulant}\ 

\begin{enumerate}
\item \label{prop:cumulant1}$K_p (0) = 0$; otherwise, for each $u \neq 0$, $K_p (u) > 0$.
\item \label{prop:cumulant2}$K_p$ is convex and lower semi-continuous, and its proper domain is a convex set which contains the open unit ball of $\eBspace p$; in particular the interior the proper domain is a non empty open convex set denoted $\sdomain p$.
\item \label{prop:cumulant3}$K_p$ is infinitely G\^ateaux-differentiable in the interior of its proper domain.
\item \label{prop:cumulant4}$K_p$ is bounded, infinitely Fr\'echet-differentiable and analytic on the open unit ball of $\eBspace p$.
\end{enumerate}
\end{proposition}

Other properties of the functional $K_p$ are described below as they relate directly to the exponential manifold.

\section{Exponential manifold}\label{exponentialmanifold}

The set of positive densities $\pdensities$ around a given $p \in \pdensities$ is modelled of the subspace of centered random variables in the Orlizc space $L^{\Phi}(p)$, hence it is crucial to discuss the isomorphism of the model spaces for different $p$'s.

\begin{definition}[Maximal exponential model: {\cite[Def 20]{cena|pistone:2007}}]
For $p \in \pdensities$ let $\sdomain p$ be the topological interior of the proper domain of the cumulant functional $K_p \colon \eBspace p$. The \emph{maximal exponential model} at $p$ is 

\begin{equation*}
  \maxexp p = \setof{\euler^{u - K_p(u)}\cdot p}{u \in \sdomain p}.
\end{equation*}
\end{definition}

\begin{definition}[Connected densities]\label{def:smile}
Densities $p,q \in \pdensities$ are connected by an open exponential arc, $p \smile q$, if there exists an open exponential family containing both, i.e. if for a neighborhood $I$ of $[0,1]$

\begin{equation*}
  \int p^{1-t}q^t \ d\mu = \expectat p {\left(\frac qp\right)^t} < +\infty, \quad t \in I.
\end{equation*}
\end{definition}

The following example is of interest for the applications in Sec. \ref{sec:applications}. Let $f_0$ be the standerd normal density on $\reals^N$ and $f$ a density, $f(x) \propto (1+\absoluteval x^a)f_0(x)$, $a > 0$. Then $\int (1+\absoluteval x^a)^t f_0(x) \ dx < + \infty$ for all real $t$, hence $ f_0 \smile f$.

\begin{proposition}[Characterization of a maximal exponential model: {\cite[Th 19 and 21]{cena|pistone:2007}}\label{prop:maxexp-pormanteaux}]
The following statement are equivalent:
\begin{enumerate}
\item \label{prop:maxexp-pormanteaux-1} $p, q \in \pdensities$ are connected by an open exponential arc, $p \smile q$;
  \item \label{prop:maxexp-pormanteaux-2} $q \in \maxexp p$;
  \item \label{prop:maxexp-pormanteaux-3} $\maxexp p = \maxexp q$;
  \item \label{prop:maxexp-pormanteaux-4} $\log \frac q p$ belongs to both $L^{\Phi}(p)$ and $L^{\Phi}(q)$;
  \item \label{prop:maxexp-pormanteaux-5} $L^{\Phi}(p)$ and $L^{\Phi}(q)$ are equal as vector spaces and their norms are equivalent.
\end{enumerate}
\end{proposition}

\begin{definition}[Exponential manifold: \cite{pistone|sempi:95,pistone|rogantin:99,cena:2002,cena|pistone:2007}]
For each $p \in \pdensities$ define the charts

\begin{equation*}
  s_p \colon \maxexp p \ni q \mapsto \lnof{\frac qp} - \expectat p {\lnof{\frac qp}} \in \sdomain p \subset \eBspace,
\end{equation*}
with inverse

\begin{equation*}
  s_p^{-1} = e_p \colon \sdomain p \ni u \mapsto \euler^{u - K_p(u)} \cdot p \in \maxexp p \subset \pdensities.
\end{equation*}
The atlas $\setof{s_p \colon \sdomain p}{ p \in \pdensities}$ is affine and defines the \emph{exponential (statistical) manifold} $\pdensities$. 
\end{definition}

The affine manifold we have defined has a simple and natural structure because of Prop.  \ref{prop:maxexp-pormanteaux}. The domains $\maxexp p$, $\maxexp q$ of the charts $s_p$, $s_q$ are either disjoint or equal when $p \smile q$:

\begin{equation}
\label{eq:scheme1}
\xymatrix{%
\maxexp p \ar[r]^{s_p}\ar@{=}[d] & \sdomain p \ar[d]_{s_q\circ s_p^{-1}} \ar[r]^{I} & \eBspace p \ar[d]^{d(s_q\circ s_p^{-1})} \ar[r]^{I} & L^{\Phi}(p) \ar@{=}[d] \\ 
\maxexp q \ar[r]_{s_q} & \sdomain q \ar[r]_{I} & \eBspace q \ar[r]_{I} & L^{\Phi}(q)
}
\end{equation}

For ease of reference, various results from \cite{pistone|sempi:95,pistone|rogantin:99,cena:2002,cena|pistone:2007} are collected in the following proposition.  We assume $q = \euler^{u - K_p(u)} \cdot p \in \maxexp p$. Note that $K_p(u) = \expectat p {\lnof{p/q}} = \KL pq$.
\begin{proposition}\label{pr:misc}

  \begin{enumerate}
\item \label{item:firsttwo} The first three derivatives of $K_p$ on $\mathcal S_p$ are

\begin{align}
d K_p(u) v &= \expectat q v,  \label{eq:Kprime}\\
d^2 K_p(u)(v_1, v_2) &= \covat q {v_1}{v_2}. \label{eq:Kdoubleprime} \\
d^3 K_p(u)(v_1, v_2,v_3) &= \Cov_q(v_1,v_2,v_3). \label{eq:Ktripleprime}
\end{align}
\item The random variable $\frac q p -1$ belongs to $\mBspace p$ and

\begin{equation*}
  d {K_p(u)} v = \expectat p {\left( \frac q p - 1\right) v}.
\end{equation*}
In other words the gradient of $K_p$ at $u$ is identified with an element of $\mBspace p$, denoted by $\nabla K_p(u) = e^{u - K_p(u)} - 1=\frac q p -1$.
\item
The mapping $B_p \ni u\mapsto \nabla K_p(u) \in \mBspace p$ is monotonic, in particular one-to-one.
\item \label{item:weakderiv}
The weak derivative of the map $\mathcal S_p \ni u\mapsto \nabla K_p(u)
  \in \mBspace p$ at $u$ applied to $w \in B_p$ is given by

\begin{equation*}
  d (\nabla K_p(u)){w}= \frac q p \left( w- \expectat q w \right),
\end{equation*}
and it is one-to-one at each point.
\item \label{pr:misc-upq} The mapping $\mtransport pq : v \mapsto \frac pq \, v$ is an isomorphism of $\mBspace p$ onto $\mBspace q$.
\item \label{pr:misc-1} $q/p \in L^{\Phi_*}(p)$.
  \item \label{pr:misc-2} $D \left(q \Vert p \right) = D K_p(u)  u - K_p(u)$ with $q=\euler^{u-K_p(u)}p$,  in particular $- \KL q p < +\infty$. 
  \item  \label{pr:misc-3} $B_q$ is defined by an orthogonality property:

    \begin{equation*}
    B_q = L_0^{\Phi}(q) = \setof{u \in L^{\Phi}(p)}{ \expectat p {u \frac qp}  = 0}.  
    \end{equation*}
  \item \label{pr:misc-5} $\etransport pq \colon u \mapsto u - \expectat q u$ is an isomorphism  of $B_p$ onto $B_q$.
  \end{enumerate}
\end{proposition}
\subsection{Tangen bundle}
\label{sec:tangent-bundle}

Our discussion of the tangent bundle of the exponential manifold is based on the concept of velocity of a curve as in \cite[\S 3.3]{abraham|marsden|ratiu:1988} and it is mainly intended to underline its statistical interpretation, which is obtained by identifying curves with one-parameter statistical models. For a statistical model $p(t)$, $t \in I$, the random variable $\dot p(t)/p(t)$, the \emph{Fisher score}, has zero expectation with respect to $p(t)$, and its meaning in the exponential manifold is velocity. If $p(t) = \euler^{tv-\psi(t)}\cdot p$, $v \in L^{\Phi}(p)$, is an exponential family, then $\dot p(t)/p(t) = v - \expectat {p(t)} v \in \eBspace {p(t)}$, see \cite{brown:86} on exponential families.

Let $p(\cdot) \colon I \to \maxexp p$, $I$ open real interval containing 0. In the chart centered at $p$ the curve is $u(\cdot) \colon I \to \eBspace p$, where $p(t) = \euler^{u(t)-K_p(u(t))}\cdot p$. The transition maps of the exponential manifold are

\begin{equation*}
  s_{q} \circ e_{p} \colon \sdomain{p} \ni u \mapsto s_{q}(\euler^{u - K_{p}(u)} \cdot p) = u - \expectat {q} u + \lnof{\frac{p}{q}} - \expectat {q} {\lnof{\frac{p}{q}}} \in \sdomain{q} = \sdomain{p},
\end{equation*}
with derivative

\begin{equation*}
  d_{v} s_{q} \circ s^{-1}_{p} (u) = v - \expectat q v = \etransport pq v, \quad v \in \eBspace p.  
\end{equation*}
\begin{definition}[Velocity field of a curve]\  
\begin{enumerate}
\item\label{item:velocity}
Assume $t \mapsto u(t) = s_p(p(t))$ is differentiable with derivative $\dot u(t)$. Define 

\begin{equation*}
  \delta p(t) = \etransport p {p(t)} \dot u(t) = \dot u(t) - \expectat {p(t)} {\dot u(t)} = \derivby t {(u(t) - K_{p}(u(t))} = \derivby t {\lnof{\frac{p(t)}{p}}} = \frac{\derivby t p(t)}{p(t)}.
\end{equation*}
Note that $\delta p$ does not depend on the chart $s_p$ and that the derivative of $t \mapsto p(t)$ in the last term of the equation is computed in $L^{\Phi_*}(p)$. The curve $t \mapsto (p(t),\delta p(t))$ is the \emph{velocity field} of the curve.
\item On the set $\setof{(p,v)}{p \in \pdensities, v \in \eBspace p}$ the charts

  \begin{equation*}
    s_p \colon \setof{(q,w)}{q \in \maxexp p, v \in \eBspace q} \ni (q,w) \mapsto (s_p(q),\etransport q p w) \in \sdomain p \times \eBspace p \subset \eBspace p \times \eBspace p
  \end{equation*}
define the \emph{tangent bundle} $T\pdensities$. The  isomorphism $w \mapsto \etransport p q w = w - \expectat p w = d (s_q\circ s_p^{-1})(u) w$ of Prop. \ref{pr:misc}(\ref{pr:misc-5}) is the (exponential) \emph{parallel transport}.
\end{enumerate}
\end{definition}

Let $E \colon \maxexp p \to \reals$ be a $C^1$ function. Then $E_p = E \circ e_p \colon \sdomain p \to \reals$ is differentiable and

\begin{equation*}
  \derivby t E(p(t)) = \derivby t E_p(u(t)) = dE_p(u(t)) \dot u(t) = dE_p(u(t)) \etransport {p(t)} p \delta p(t).
\end{equation*}

\begin{proposition}[Covariant derivative of a real function]\ 
\label{prop:covariant-derivative}%
\begin{enumerate}
\item
As $v \mapsto dE_p(u) v$ is a linear operator on $\eBspace p$, $w \mapsto dE_p(u) \etransport {e_p(u)} p w$ is a linear operator on $\eBspace {e_p(u)}$ which does not depend on $p$. 
\item \label{item:covariant-derivative2}
If $G$ is a vector field in $T\pdensities$, the \emph{covariant derivative} $D_GE$ is

\begin{equation*}
  D_G E (q) = dE_p(s_p(q)) \etransport {e_p(u)} p w  = dE_q(0) w, \quad w = G(q).
\end{equation*}
\item\label{item:covariant-derivative-3} Assume moreover that $dE_p(u) \in {\eBspace p}^*$ can be identified with an element $\nabla E_p(u) \in \mBspace p$ by

  \begin{equation*}
    dE_p(u) w = \expectat p {\nabla E_p(u) w}, \quad w \in \eBspace p.
  \end{equation*}
Then for $u = e_p(q)$

\begin{equation*}
    D_GE(q) = dE_p(u)\etransport qp G(q) = \expectat q {\mtransport pq {\nabla E_p(u)} G(q)}.
\end{equation*}
We define the \emph{covariant gradient} $\nabla_GE(q)$ to be defined by $D_GE(q) = \expectat q {\nabla_GE(q) G(q)}$. 
\end{enumerate}
\end{proposition}
\begin{proof}\ 
\begin{enumerate}
\item
Assume $u_1 = s_{p_1}(q) = s_{p_1}\circ e_{p_2}(u_2)$ so that $E(q) = E_{p_1}(u_1) = E_{p_2}(u_2) = E_{p_1}(s_{p_1}\circ e_{p_2}(u_2))$ and

\begin{equation*}
  dE_{p_2}(u_2) \etransport {q} {p_2} w = dE_{p_1}\circ s_{p_1}\circ e_{p_2}(u_2) \etransport {q} {p_2} w = dE_{p_1}(u_1) \etransport {p_1} {p_2} \etransport {q} {p_2} w = dE_{p_1}(u_1) \etransport {q} {p_1} w.
\end{equation*}
\item Compute the derivative of $t \mapsto E\circ p(t)$ when $\delta p(t) = G(p(t))$. 
\item It is a computation based on

  \begin{equation*}
    \expectat p {\nabla E_p(u)(G(p) - \expectat p {G(q)})} = \expectat q {\frac pq \nabla E_p(u)(G(q) - \expectat p {G(q)})} = \expectat q {\frac pq \nabla E_p(u)(G(q)}.
  \end{equation*}
\end{enumerate}
\end{proof}
\begin{definition}
  Let $F, G \colon \maxexp p$ be  vector fields of $T\pdensities$. In the chart at $p$, $F_p(u) = \etransport {e_p(u)} p F \circ e_p(u)$, $u \in \sdomain p$ has differential $B_p \colon v \mapsto d F_p(u) v \in B_p$. The \emph{e-covariant derivative} is the vector field defined by $D_GF(q) = \etransport p{q} d F_p(s_p(q)) \etransport {q} p w$, $w = G(q)$, and this definition does not depend on $p$.  
\end{definition}

\subsection{Pretangent bundle}
\label{sec:pretangent-bundle}

Because of the lack of reflexivity of the exponential Orlicz space, we are forced to distinguish between the dual tangent bundle $(T\pdensities)^* = \setof{(p,v)}{p \in \pdensities, v \in (\eBspace p)^*}$ and a pretangent bundle.
\begin{definition}
The set $\setof{(q,v)}{q \in \pdensities, v \in \mBspace q}$ together with the charts

\begin{equation*}
  \prescript{*}{}s_p \colon \setof{(q,v)}{q \in \maxexp p, v \in \mBspace q} \ni (q,v) \mapsto \left(s_p(q), \mtransport q p v\right)
\end{equation*}
is the \emph{pretangent bundle} $\prescript{*}{}T\pdensities$.
\end{definition}

The pretangent bundle is actually the tangent bundle of the \emph{mixture manifold} on $\sdensities = \setof{f \in L^1(\mu)}{\int f \ d\mu = 1}$ whose charts are of the form $\eta_p(q) = q/p -1 \in \mBspace p$. For each $p \in \pdensities$ consider the set

\begin{equation*}
  \mcoverat p = \setof{q \in \sdensities}{q/p \in L^{\Phi_*}(p)}
\end{equation*}
and the mapping

\begin{equation*}
  \eta_p \colon \mcoverat p \ni q \mapsto \eta_p(q) = q/p-1 \in \mBspace p. 
\end{equation*}

Let characterize $\mcoverat p$ as the set the set of $q$'s of finite Kullback-Leibler divergence from $p$. 
\begin{proposition}[{\cite[Prop. 30]{cena:2002}}]
Let $p \in \pdensities$ and $q \in \sdensities$. Define $\tilde q = \absoluteval q / \int \absoluteval q \ d\mu$. Then $\KL{\tilde q} p < +\infty$ if, and only if, $q/p \in L^{\Phi_*}(p)$.   
\end{proposition}
\begin{proof}
The second derivative of $\Phi_*(x) = (1+x)\lnof{1+x} -x$, $x>0$, is $1/(1+x)$, while the second derivative of $x\lnof x$ is $1/x$. The function $x\lnof x$ is more convex that $\Phi_*(x)$ as $0 < 1/(1+x) < 1/x$. The two functions have parallel tangents at $x > 0$ if $\lnof{1+x} = \lnof x + 1$, that is at $\bar x = 1/(\euler -1)$. At this point the difference of the values is

\begin{equation*}
\Phi_*(\bar x) - \bar x \lnof{\bar x} = 1 - \lnof{\euler-1}.   
\end{equation*}
In conclusion, we have the inequalities

\begin{equation}\label{eq:phistarent}
  \Phi_*(x) \le x\lnof x + 1 - \lnof{\euler-1} < x\ln x + 1 , \quad x \ge 0.
\end{equation}
If $\KL{\tilde q}p < + \infty$. then

\begin{equation*}
  + \infty > \int \lnof{\frac{\tilde q}{p}} \tilde q \ d\mu = \expectat p {\frac{\tilde q}{p} \lnof{\frac{\tilde q}{p} }} > \expectat p {\Phi_*\left(\frac{\tilde q}{p}\right)} -1 = \expectat p {\Phi_*\left(\left(\int \absoluteval q \ d\mu \right)^{-1}\frac{q}{p}\right)} -1, 
\end{equation*}
so that $q/p \in L^{\Phi_*}(p)$.

Assume now $q/p \in L^{\Phi_*}(p)$, or, equivalently, $\tilde q / p \in L^{\Phi_*}(p)$. As $x\ln^+(x) \le (1+x)\lnof x$ for $x \le 0$, we have

\begin{equation*}
+\infty > \expectat p {\phi_*\left(\frac {\tilde q}p\right)} = \expectat p {\left(1+\frac {\tilde q}p\right)\lnof{1+\frac {\tilde q}p}} - 1 \ge \expectat p {\frac{\tilde q}p \ln^+\left(\frac{\tilde q}p\right)} - 1,
\end{equation*}
which in turn implies that 

\begin{equation*}
  \KL{\tilde q}p < \expectat p {\frac{\tilde q}p \ln^+\left(\frac{\tilde q}p\right)}
\end{equation*}
is finite.
\end{proof}
 
The covariant gradient defined in Prop. \ref{item:covariant-derivative-3}(\ref{item:covariant-derivative-3}) is a vector field of the pretangent bundle. Note that the injection $\pdensities \hookrightarrow \sdensities$ is represented in the charts centered at $p$ by $u \mapsto \euler^{u-K_p(u)} \cdot p - 1$. We do not further discuss here the mixture manifold  and refer to \cite[Sec. 5]{cena|pistone:2007} for further information on this topic. 

Let $F$ be a vector field of the pretangent bundle $\prescript{*}{}T\pdensities$. In the chart centered at $p$ $\maxexp p \ni q \mapsto F(q)$ is represented by

  \begin{equation*}
    F_p(u) = \mtransport {e_p(u)} p {F\circ e_p (u)} \in \mBspace p, \quad u \in \sdomain p. 
  \end{equation*}
If $F_p$ is of class $C^1$ with derivative $dF_p(u) \in L(\eBspace p, \mBspace p)$, for each differentiable curve $t \mapsto p(t) = \euler^{u(t) - K_p(u(t))} \cdot p$,

\begin{equation*}
  \derivby t F_p(p(t)) = d F_p(u(t)) \dot u(t) = dF_p(u(t)) \etransport {p(t)} p {\delta p(t)} \in \mBspace p.  
\end{equation*}
For each $q = e_p(u) \in \maxexp p$, $w \in \mBspace q$, $\mtransport p q {dF_p(u)} \etransport q p w \in \mBspace q$ does not depend on $p$.
  \begin{definition}[Covariant derivative in $\prescript{*}T\pdensities$.]
    Let $F$ be a vector field of the pretangent bundle $\prescript{*}{}T\pdensities$ and $G$ a vector field in the tangent bundle $\pdensities$, both of class $C^1$ on $\maxexp p$. The covariant derivative is
   
 \begin{equation*}
      D_GF(q) = d \mtransport {e_q(u)} q {F\circ e_q }(0) w, \quad w = G(q).
    \end{equation*}
\end{definition}

Tangent and pretangent bundle can be coupled to produce the new frame bundle

\begin{equation*}
  (\prescript{*}{}T \times T)\pdensities = \setof{(p,v,w)}{p \in
    \pdensities, v \in \eBspace p, w \in \mBspace p}
\end{equation*}
with the duality coupling

\begin{equation*}
  (\prescript{*}{}T \times T)\pdensities \ni (p,v,w) \mapsto \scalarat p
  v w = \expectat p {uv} = \expectat q {\mtransport p q v \etransport
    p q w}, \quad p \smile q.
\end{equation*}

\begin{proposition}[Covariant derivative of the duality coupling]
Let $F$ be a vector field of $\prescript{*}{}T\pdensities$, $G, H$ vector fields of $T\pdensities$, all of class $C^1$ on a maximal exponential model $\mathcal E$. Then

\begin{equation*}
  D_H \scalarof F G = \scalarof {D_H F} G + \scalarof F {D_HG}.
\end{equation*}
\end{proposition}
\begin{proof}
Consider the real function $\mathcal E \ni q \mapsto \scalarof F G (q)
= \expectat q {F(q)G(q)}$ in the chart centered at any $p \in \mathcal
E$,

\begin{equation*}
  \sdomain p  \ni u \mapsto \expectat q {F(q)G(q)} = \expectat p
  {\mtransport q p {F \circ e_p(u)} \etransport q p {G \circ e_p(u)}}
  = \expectat p {F_p(u)G_p(u)} 
\end{equation*}
and compute its derivative.
\end{proof}
\subsection{The Hilbert bundle.}
\label{sec:hilbert-bundle}
The duality on $(\prescript{*}{}T \times T)\pdensities$ is reminiscent of a Riemannian metric, but it is not, because we do not have a Riemannian manifold unless the state space is finite. However, we we can push on the analogy, by constructing an Hilbert bundle. As $L^{\Phi}(p) \subset L^{2}(p) \subset L^{\Phi_*}(p)$, $p \in \pdensities$,  we have $\eBspace p \subset H_p \subset \mBspace p$, $L^2_0(p) = H_p$ being the fiber at $p$. The Hilbert bundle

\begin{equation*}
  H \pdensities = \setof{(p,v)}{p \in \pdensities, v \in H_p}
\end{equation*}
is provided with an atlas of charts by using the isometries $\transport pq \colon H_p \to H_q$  which result from the pull-back of the metric connection on the sphere $S_\mu = \setof{f \in L^2(\mu)}{\int f^2 \ d\mu = 1}$, see \cite{gibilisco|pistone:98,gibilisco|isola:1999,grasselli:2010AISM} and \cite[Sec. 4]{pistone:2013GSI2013}.

\begin{proposition}[Isometric transport: {\cite[Prop. 13]{pistone:2013GSI2013}}]\ 
\label{prop:Tisometry}
 \begin{enumerate}
  \item For all $p,q \in \pdensities$, the mapping

 \begin{equation*}
\transport pq \colon v \mapsto \sqrt{\frac pq} u  -  \left(1 +  \expectat q {\sqrt{\frac pq}}\right)^{-1} \left(1 + \sqrt{\frac pq}\right) \expectat q {\sqrt{\frac pq}v}      
    \end{equation*}
is an isometry of $H_p\pdensities$ onto $H_q\pdensities$. 
  \item $\transport qp \circ \transport pq u = u$, $u \in H_p\pdensities$ and $(\transport pq)^t = \transport qp$.
  \end{enumerate}
\end{proposition}
Note that $\transport qr \transport pq \neq \transport pr$.
\begin{definition}[Hilbert bundle]
The charts

\begin{equation*}
  \prescript{2}{}s_p \colon \setof{(q,v)}{q \in\maxexp p, v \in H_q} \ni (q,v) \mapsto \left(s_p(q),\transport qp v\right) \in \sdomain p \times H_p \subset \eBspace p \times H_p
\end{equation*}
form an atlas on $H\pdensities$.
\end{definition}

Let $t \mapsto p(t)$ be a $C^1$ curve in $\maxexp p$, $p = p(0)$, $u(t) = s_p(p(t))$, and $F \colon \maxexp p$ a $C^1$ vector field in $H\pdensities$. In che chart centered at $p$ we have $F_p(u(t)) = \transport {p(t)} p (F \circ e_p)(u(t))$. A computation shows that

\begin{align*}
  \left. \derivby t F_p(t) \right|_{t=0} &= \left. \derivby t \transport {p(t)} p (F \circ e_p)(u(t)) \right|_{t=0} \\ &= dF_p(0) \delta p(0) + \frac12 F_p(0) \delta p(0) - \expectat p {dF_p(0) \delta p(0) + \frac12 F_p(0) \delta p(0)}, 
\end{align*}
which could be used as a nonparametric definition of the metric connection, see \cite{grasselli:2010AISM}, \cite[Sec. 4.4]{pistone:2013GSI2013}.

\subsection{The second tangent bundle}

We briefly discuss here the second order structure, i.e. the tangent bundle of tangent bundle $T\pdensities$. Let $F \colon I \ni t \mapsto (p(t),V(t))$ be a $C^1$ curve in the tangent bundle $T\pdensities$. In the chart centered at $p$ we have

\begin{equation*}
  F_p(t) = (s_p(p(t)),\etransport {p(t)} p V(t)) = (u(t), V_p(t)),
\end{equation*}
where $p(t) = \euler^{u(t) - K_p(u(t))} \cdot p$ and $V(t) = V_p(t) - \expectat {p(t)} {V_p(t)} = V_p(t) - dK_p(u(t))(V_p(t))$. It follows that $t \mapsto V(t)$ in differentiable in $L^{\Phi}(p)$, with derivative

\begin{equation*}
  \dot V(t) = \dot V_p(t) - dK_p(u(t))(\dot V_p(t)) - d^2(p(t))(V_p(t),\dot u(t)) = \etransport  p{p(t)}{\dot V_p(t)} - \covat {p(t)} {V_p(t)}{\dot u(t)},
\end{equation*}
hence

\begin{equation}\label{eq:etransportVdot}
 \etransport  p{p(t)}{\dot V_p(t)} = \dot V(t) + \expectat {p(t)} {V(t)}{\delta p(t)}.
\end{equation}
It follows in particular that $\expectat {p(t)}{\dot V(t)} = - \expectat {p(t)} {V(t)\delta p(t)}$ and $\etransport  p{p(t)}{\dot V_p(t)} = \dot V(t) - \expectat {p(t)} {\dot V(t)}$. Note that the left end side is not a transport but an extension of the transport, precisely the projection $\Pi^{p(t)} \colon L^{\Phi}(p) \to \eBspace {p(t)}$. It follows from $  \dot F_p(t) = \left(\dot u(t), \dot V_p(t)\right)$ that the velocity vector is

\begin{equation*}
\delta (p,V)(t) = \left(\delta p(t), \etransport p {p(t)} {\dot V_p(t)}\right) = \left(\delta p(t), \Pi^{p(t)} {\dot V(t)}\right).
\end{equation*}

The equality \eqref{eq:etransportVdot} in the case $V(t) = (\delta p)(t)$ gives

\begin{equation*}
  \Pi^{p(t)} {\dot {(\delta p)}(t)} = \dot {(\delta p)}(t) + \expectat {p(t)} {\delta p(t)^2} = \dot {(\delta p)}(t) + I(p(t)),  
\end{equation*}
where we have denoted by $I(p(t)) = \expectat {p(t)}{\derivby t {\lnof {p(t)}^2}}$ the Fisher information. In this case we can write

  \begin{equation*}
    \delta (p,\delta p)(t) = (\delta p (t), \dot{(\delta p)}(t) + I(p(t))). 
  \end{equation*}

\section{Applications.}
\label{sec:applications}

In this section we consider a typical set of exemples where the nonparametric framework is applicable.

\subsection{Expected value.}
\label{sec:expected-value}

Let $f \in L^{\Phi}(p)$, $f_0 = f - \expectat p f \in \eBspace p$, and consider the relaxed mapping

\begin{equation}\label{eq:relaxed}
  E \colon\maxexp p \ni q \mapsto \expectat q f = \expectat q {f_0} + \expectat p f.
\end{equation}
The information geometric study of the relaxed mapping can be based on the notion of \emph{natural gradient} as defined in a seminal paper by Amari \cite{amari:1998natural} and it is currently used for optimization, see e.g. \cite{malago|matteucci|dalseno:2008,malago|matteucci|pistone:2009NIPS, malago|matteucci|pistone:2011a,Wierstra:arXiv1106.4487,arnoldetal:2011arXiv,malago:2012thesis,malago|matteucci|pistone:2013CEC}. Covariant derivative of a real function is the nonparametric counterpart of Amari's natural gradient.

From the properties of $K_p$ in Eq.s \eqref{eq:Kprime} and \eqref{eq:Kdoubleprime} of Prop. \ref{pr:misc} we obtain the representation of the function in \eqref{eq:relaxed} in the chart centered at $p$, $E_p(u) = dK_p(u)(f_0) + \expectat p f$ and its differential $d E_p(u) v = d^2K_p(u)(f_0,v) = \covat q {f}{v}$. The covariant derivative at $(q.w) \in T\eBspace q$ is computed from Def. \ref{prop:covariant-derivative}(\ref{item:covariant-derivative2}) as 

\begin{equation*}
  d E_p(u) \etransport {q}{p} w = \covat q {f}{\etransport {q}{p} w} = \expectat q {(f-\expectat q f)(w - \expectat p w)} = \expectat q {(f-\expectat q f) w},
\end{equation*}
hence $D_{G} E(p) = \expectat p {(f - \expectat p f)G(p)}$ with gradient $\nabla_{G} E(q) = f - \expectat q f$ in the duality on $\mBspace q \times \eBspace q$. Note that the gradient is never zero unless $f$ is constatant and that the covariant derivative is zero for each vector field $G$ which is uncorrelated with $f$. 

Consider the gradient vector field $F(q) = f - \expectat q f \in \prescript{*}{}T\pdensities$. The gradient flow is

\begin{equation*}
  \delta p(t) = \derivby t \lnof{p(t)} = f - \expectat {p(t)} f,
\end{equation*}
whose unique solution is the exponential family $p(t) \propto \euler^{tf} \cdot p(0)$. In fact, the gradient is actually the e-transport of $f_0$, $F(p) = \etransport p {e_p(u)} f_0$ and the exponential family is the exponential curve of the e-transport.

Let us discuss the differentiability of the gradient. In the chart centered at $p$ the gradient is represented as

\begin{equation*}
F_p(u) = \mtransport {e_p(u)} p [f_0 - d K_p(u) (f_0)] = \mtransport {e_p(u)} p \etransport p {e_p(u)} f_0.
\end{equation*}
Let us first compute the differential of $u \mapsto \scalarat p {F_p(u)}{w}$, $w \in B_p$, in the direction $v \in B_p$, i.e. the weak differential:

\begin{multline*}
  d_v \scalarat p {F_p(u)}{w} = d_v \scalarat p {\mtransport {e_p(u)} p \etransport p {e_p(u)} f_0}{w} = \covat {e_p(u)} {f_0} w = \\ d_v d^2 K_p(u)(f_o,w) = d^3 K_p(u)(f_0,w,v) = \Cov_{e_p(u)}(f_0,w,v), 
\end{multline*}
where we have used Prop. \ref{pr:misc}. At $u = 0$

\begin{equation*}
  d_v \scalarat p {F_p(0)}{w} = \expectat p {f_0wv} = \expectat p {\left(f_0v - \expectat p {f_0v}\right)w} = \scalarat p {f_0v - \expectat p {f_0v}} w.
\end{equation*}

The product $f_0G(p)$ belongs to $\mBspace p$. In fact, 

\begin{equation*}
  \expectat p {\Phi_*\left(f_0G(p)\right)} = \expectat p {f_0^2 \integrali 0 {\absoluteval{G(p)}} {\frac{\absoluteval{G(p)} - u}{1+\absoluteval{f_0}u}} u } \le \frac 12 \expectat p {f_0^2 G(p)^2} < + \infty.
\end{equation*}
If $D_G\nabla E$ exists in $\prescript{*}{}\pdensities$ as a Frech\'et derivative, then

\begin{equation*}
  D_{G}\nabla E(p) = f_0G(p) - \expectat p  {f_0G(p)}.
\end{equation*}
The differentiability in Orlicz spaces of superposition operators is discussed in detail in \cite{appell|zabrejko:1990}.

\subsection{Kullback-Leibler divergence.}
\label{sec:kullb-leibl-diverg}

If $\mathcal E$ is a maximal exponential model, the mapping

\begin{equation*}
  \mathcal E \times \mathcal E \ni (q_1,q_2) \mapsto \KL {q_1}{q_2} = \expectat {q_1} {\lnof{\frac {q_1}{q_2}}} 
\end{equation*}
is represented in the charts centered at $p$ by

\begin{equation*}
  E_p \colon \sdomain {p} \times \sdomain {p} \ni (u_1,u_2) \mapsto dK_p(u_1)(u_1 - u_2) - (K_p(u_1) - K_p(u_2)), 
\end{equation*}
hence, from Prop. \ref{prop:cumulant}(\ref{prop:cumulant4}) it is $C^\infty$ jointly in both variables, and moreover analytic

\begin{equation*}
  E_p(u_1,u_2) = \sum_{n \ge 2} \frac 1{n!}d^n K_p(u_1) (u_1 - u_2)^{\circ n}, \quad \normat {\Phi,p} {u_1 - u_2} < 1. 
\end{equation*}
This regularity result is to be compared with what is available when the restriction $q_1 \smile q_2$ is removed, i.e. the semicontinuity \cite{ambrosio|gigli|savare:2008}.

The (partial) derivative of $u_2 \mapsto E_p(u_1,u_2)$ in the direction $v_2 \in \eBspace p$ is 

\begin{equation*}
  d_2 E_p(u_1,u_2) v_2 = -d K_p(u_1) v_2 + d K_p(u_2) v_2 = \expectat {q_2} {v_2} - \expectat {q_1} {v_2}.
\end{equation*}
If $v_2 = \etransport q p w$, we have $\expectat {q_2} {v_2} - \expectat {q_1} {v_2} = \expectat {q_2} {w} - \expectat {q_1} {w}$ and the covariant derivative of the partial functional $q \mapsto \KL {q_1} q$ is

\begin{equation*}
  D_{2,w} \KL {q_1} q = \expectat {q} {w} - \expectat {q_1} {w} = \expectat q {\left(1 - \frac {q_1}q\right) w}, \quad \nabla_q \KL {q_1} q = 1 - \frac {q_1} q
\end{equation*}
The second mixed derivative of $E_p$ is

\begin{equation*}
  d_1 d_2 E_p(u_1,u_2) (v_1,v_2) = - d^2 K_p (u_1) (v_1,v_2) = - \covat {q_1} {v_1}{v_2}.
\end{equation*}

Equivalently, we consider the mapping $q_1 \mapsto  D_{2,w} \KL {q_1} q$, in the chart $u_1 \mapsto \expectat {q} {w} - \expectat {q_1} {w}$, to obtain

\begin{equation*}
  \left. D_{1,w_1} D_{2,w_2} \KL {q_1}{q_2} \right|_{q_1 = q_2 = q} = - \expectat q {w_1,w_2}.
\end{equation*}

\subsection{Boltzmann entropy.}
\label{sec:boltzmann-entropy}

While our discussion of the Kulback-Leibel divergence in the previous Sec. \ref{sec:kullb-leibl-diverg} does not require any special assumption but the restriction of its domain to a maximal exponential model, in the present discussion of the Boltzmann entropy a further restriction is required. If $p, q$ belong to the same maximal exponential model, $p \smile q$, then from $q = \euler^{u - K_p(u)} \cdot p$ with $u \in \eBspace p$, we obtain $\ln q - \ln p \in L^{\Phi}(p)$, so that $\ln q \in L^{\Phi}(p)$ if, and only if, $\ln p \in L^{\Phi}(p)$.

We study the Boltzmann entropy $E(q) = \expectat q {\lnof q}$ on a maximal exponential model $q\in \mathcal E$ such that for at least one, and hence for all, $p \in \mathcal E$ it holds $\lnof p \in L^{\Phi}(p)$, i.e. $\int \left(p^{1+\alpha}+p^{1-\alpha}\right)\ d\mu < +\infty$ for some $\alpha >0$. This is for example the case when the reference measure is finite and $p$ is constant. Another notable example is the Gaussian case, i.e. the sample space is $\reals^n$ endowed with the Lebesgue measure and $p(x) \propto \exp{-1/2 |x|^2}$. In fact $\int \cosh (\alpha |x|^2) \expof{-1/2 |x|^2} \ dx < +\infty$ for $0 < \alpha < 1/2$.

Under our assumption, the Boltzmann entropy is a smooth function. As

\begin{equation*}
  \lnof q = u - K_p(u) + \lnof p = u - K_p(u) + (\lnof p - E(p)) + E(p) \in L^{\Phi}(p),
\end{equation*}
the representation in the chart centered at $p$ is

\begin{equation*}
  E_p(u) = \expectat {e_p(u)}{u - K_p(u) + \lnof p} = dK_p(u) \left[u + (\lnof p - E(p))\right] - K_p(u) + E(p),
\end{equation*}
hence it is a $C^\infty$ real function. The derivative in the direction $v$ equals

\begin{equation*}
  dE_p(u) v = d^2 K_p(u) \left(u + (\lnof p - E(p)),v \right)= \covat q {u+\lnof p}v,
\end{equation*}
in particular

\begin{equation*}
  dE_p(0)v = \expectat p {(\lnof p - E(p)) v} = \scalarat p {\lnof p - E(p)}{v}.
\end{equation*}
The value of the covariant derivative $D_GE$ at $q$ and $G(q) = w$ is

\begin{equation*}
  dE_p(u) \etransport q p w =  \covat q {u+\lnof p}{w} = \expectat q {((\lnof q  + K_p(u))w} = \expectat q {(\lnof q - E(q)) w}.
\end{equation*}

The gradient $\nabla E(q) \in (\eBspace q)^*$, $D_GE(q) = \scalarat q {\nabla E(q)}{G(q)}$, is identified with a random variable in $\eBspace q \subset \mBspace q$, and

\begin{align*}
  F(q) &= \lnof q - E(q) \\
&=  u - K_p(u) + \lnof p - \expectat q{u - K_p(u) + \lnof p} \\
&=(u  + \lnof p - E(p)) - dK_p(u)(u + \lnof p - E(p)) \\
&= \etransport p q (u  + \lnof p - E(p)) \in \eBspace q
\end{align*}
is a vector field in the tangent bundle $T\mathcal E$, hence a vector field in the Hilbert bundle $H\mathcal E$ and in the pretangent bundle $\prescript{*}{}T\mathcal E$. 

The equation $\nabla E(q) = 0$ implies $q = E(q)$, hence constant. The Boltzmann entropy is increasing along the vector field $G \in T\mathcal E$ if $\expectat q {(\lnof q - E(q))G(q)} = \covat q {\lnof q}{G(q)} > 0$. The exponential family tangent at $p$ to $\nabla E(p)$ is $p(t) \propto \euler^{t\lnof p} \cdot p = p^{1+t}$. The gradient flow equation is $\delta q(t) = \nabla E(q(t))$ that is 

\begin{equation*}
  \derivby t {\lnof{q(t)}} = \ln{q(t)} - E(q(t)).
\end{equation*}

In the pretangent bundle the action of the dual exponential transport $(\etransport qp)^*$ is identified with $\mtransport qp$. It follows that the representation of the gradient in the chart centered at $p$ is

\begin{align*}
F_p(u) &= \euler^{u-K_p(u)}\left[(u  + \lnof p - E(p)) - dK_p(u)(u + \lnof p - E(p))\right] \\ &= \mtransport {e_p(u)}p \etransport p{e_p(u)} [(u  + \lnof p - E(p)).
\end{align*}
Let us assume $u \mapsto F_p(u)$ is (strongly) differentiable and let us compute the derivative by the product rule. As $u \mapsto F_p(u)$ can be seen locally as the product of an analytic mapping $u \mapsto \euler^{u-K_p(u)}$ with values in $L^a(p)$, $a > 1$ because of Prop. \ref{prop:expisanalytic}, while the second factor is an analytic function with values in $L^{\Phi}(p) \subset \cap_{a >1} L^a(p)$, we can compute its differential in the direction $v \in \eBspace p$ as the product of two functions in the Fech\'et space $\cap_{a >1} L^a(p)$ as

\begin{multline*}
d (\nabla E)_p(u) v = \euler^{u-K_p(u)}\  \times \\ \left[(v - d K_p(u) v)\left[(u  + \lnof p - E(p)) - dK_p(u)(u + \lnof p - E(p))\right] \right. \\ \left. + v - d^2 K_p(u) (u+\lnof p-E(p),v) - dK_p(u)v\right] = \\ \frac qp \left[ (v - \expectat q v)(\lnof{q} - E(q)) + v - \expectat q v - \covat q {\lnof q}{v}\right],
\end{multline*}
in particular, for $u=0$,

\begin{align*}
d (\nabla E)_p(0) v &= (\lnof{p} - E(p) + 1)v - \expectat p {\lnof p v} \\
                    &= (\nabla E(p) + 1)v - \expectat p {\nabla E(p)v}.
\end{align*}
The covariant derivative of the gradient $\nabla E$ of the Boltzmann entropy in the pretangent bundle $\prescript{*}{}T\mathcal E$ is

\begin{align*}
  D_G(\nabla E)(p) &= (\lnof p - E(p) + 1)G(p) + \expectat p {\lnof p G(p)} \\ &= (\nabla E(p) + 1)G(p) + \expectat p {\nabla E(p)G(p)}, \quad p \in \mathcal E.
\end{align*}
The existence of the covariant derivative implies $\lnof p G(p) \in L^{\Phi_*}(p)$, $p \in \mathcal E$. We do not discuss here the existence problem.

The computation of the covariant derivative of the same gradient in the tangent bundle $T\mathcal E$ would be

\begin{align*}
  \bar F_p(u) &= \etransport qp (\lnof q - E(q)) = \lnof q - \expectat p {\lnof q} = u + \lnof p - E(p), \\
  d \bar F_p(u) v &= v,
\end{align*}
but we cannot suggest any use of this computation.
\subsection{Boltzmann equation}\label{ex:boltzmann-1} 

Orlicz spaces as a setting for Boltzmann equation has been recently discussed in \cite{majewski|labuschagne:2013}, while the use of exponential manifolds has been suggested in \cite[Example 11]{pistone:2013GSI2013}. Here we further work out this framework for space-homogeneous Boltzmann operator with angular collision kernel $B(z,x) = \absoluteval{x'z}$, see the presentation in \cite{villani:2002review}. In order to avoid a clash with the notations used in other parts of this paper, we use $v$ and $w$ to denote velocities in $\reals^3$ in place of the more common couple $v$ and $v_*$ and the velocities after collision are denoted by $v_{x}$ and $w_x$ instead of $v'$, $v'_*$, $x \in S^2$ being a unit vector.

Let $v,w \in \reals^3$ be the velocities of two particles, and $\bar v, \bar w$ be the velocities after a elastic collision, i.e.

\begin{equation}\label{eq:elastic-collision}
  v + w = \bar v + \bar w, \quad
 \absoluteval {v}^2 +   \absoluteval {w}^2 = \absoluteval {\bar v}^2 + \absoluteval {\bar w}^2.
\end{equation}

Using \eqref{eq:elastic-collision} we derive from the development of  $\absoluteval {v + w}^2 = \absoluteval {\bar v + \bar w}^2$ that $v \cdot w = \bar v \cdot \bar w$. The four vectors $v, w, \bar v, \bar w$ all lie on a circle with center $z = (v+w)/2 = (\bar v + \bar w) / 2$. In fact, the four vectors and $z$ lie on the same plane because $v - z = -(w -z)$, $\bar v -z = -(\bar w -z)$, and moreover $\absoluteval {v - z}^2 = \absoluteval {\bar v - z}^2$. As $v,w,\bar v,\bar w$ form a rectangle, we can denote by $x$ the common unit vector unit of the parallel sides $\bar w - w$ and $v -\bar v$ and write $\bar w - w = v -\bar v$ as the orthogonal projection of $v-w$ on $x$. Given the unit vector $x \in S_2 = \setof{x \in \reals^3}{x'x = 1}$, the collision transformation $(v,w) \mapsto (\bar v, \bar w) = (v_x,w_x)$ is linear and represented by a $\reals^{(3+3)\times(3+3)}$ matrix

\begin{equation}\label{eq:A-x-matrix}
  A_x =
  \begin{bmatrix}
    (I - \Pi_x) & \Pi_x \\ \Pi_x & (I - \Pi_x)
  \end{bmatrix},
\quad
\left\{\begin{aligned}
    v_x &= v - x {x'}(v - w) = (I - x {x'})v + x {x'} w , \\
    w_x &= w + x {x'}(v - w) = x {x'} v + (I - x {x'})w,
  \end{aligned}\right.
\end{equation}
where $'$ denotes the transposed vector.

Given any $x \in S_2$  we have $A_x = A_{-x}$. If $v, w, v_x, w_x$ are as in \eqref{eq:A-x-matrix} then the elastic collision invariants of \eqref{eq:elastic-collision} hold, $v + w = v_x + w_x$, $\absoluteval {v}^2 +   \absoluteval {w}^2 = \absoluteval {v_x}^2 + \absoluteval {w_x}^2$. The components in the direction $x$ are exchanged, $xx'v_x = xx'w$ and $xx'w_x = xx'v$, while the orthogonal components are conserved.

Let $\sigma$ be the uniform probability on $S^2$. For each positive function $g \colon \reals^3 \times \reals^3$ the integral $\integrald {S_2} {g(v_x,w_x)} {\sigma(dx)}$ depends on the collision invariants only. In fact,  

\begin{align*}
  v_x &= \frac{v+w}2 + \frac{\absoluteval{v-w}}2 y, \\
  w_x &= \frac{v+w}2 - \frac{\absoluteval{v-w}}2 y,
\end{align*}
where $y = \versof{v_x - w_x} = (I - 2 xx')\versof{v-w}$ and al other terms depend on the collision invariants, in particular $\absoluteval{v-w}^2 = 2(\absoluteval{v}^2+\absoluteval{w}^2)-\absoluteval{v+w}^2$. 

On the sample space $(\reals^3,dv)$ let $f_0$ be the standard normal density viz. the Maxwell distribution of velocities. As $A_xA_x = I_6$ the identity matrix on $\reals^6$, in particular $\absoluteval{\det A_x} = 1$, we have 

\begin{equation*}
  A_x (V,W) = (V_x,W_x) \sim (V,W)
\end{equation*}
 if $(V,W)$ is N$(0_6,I_6)$. We can give the previous remarks a more probabilistic form as follows.

\begin{proposition}\label{prop:conditioning}
Let $f_0$ the density of the standard normal N$(0_3,I_3)$.
\begin{enumerate}
\item \label{prop:conditioning1}
If $(V,W) \sim f_0\otimes f_0$, then $\integrald{S^2}{g(V_x,W_x)}{\sigma(dx)}$ is the conditional expectation of $g(V,W)$ given $V+W$ and $\absoluteval{V}^2+\absoluteval{W}^2$.
\item \label{prop:conditioning2}
Assume $(V,W) \sim f$, $f \in  \maxexp {f_0\otimes f_0}$,  Then $\integrald{S^2}{f\circ A_x}{\sigma(dx)} \in \maxexp {f_0\otimes f_0}$ and

\begin{equation*}
  \condexp {g(V,W)}{V+W, \absoluteval{V}^2+\absoluteval{W}^2} = \frac{\integrald{S^2}{g(V_x,W_x)f(V_x,W_x)}{\sigma(dx)}}{\integrald{S^2}{f(V_x,W_x)}{\sigma(dx)}}
\end{equation*}
\end{enumerate}
\end{proposition}

\begin{proof} \ref{prop:conditioning1}. The random variable $\integrald{S^2}{g(V_x,W_x)}{\sigma(dx)} = \integrald{S^2}{g\circ A_x(V,W)}{\sigma(dx)}$ is a function $\tilde g(m_1(V,W),m_2(V,W))$ with $m_1(V,W) = V+W$ and
$m_2(V,W) = \absoluteval{V}^2+\absoluteval{W}^2$. For all $h_1 \colon \reals^3$, $h_2 \colon \reals^3$, 

\begin{multline*}
\expectof{\left(\integrald{S^2}{g\circ A_x(V,W)}{\sigma(dx)}\right)h_1(m(V,W)h_2(m_2(V,W)))}
= \\ \expectof{g(V,W)h_1(m(V,W)h_2(m_2(V,W)))}
\end{multline*}
because of $A_x (V,W) \sim (V,W)$ and $m_1\circ A_x = m_1$, $m_2 \circ A_x = m_2$.

\ref{prop:conditioning2}. We use Prop. \ref{prop:maxexp-pormanteaux}. If $f \in \maxexp {f_0\otimes f_0}$, then 

\begin{equation*}
  f = \euler^{u - K_0(u)} \cdot f_0 \otimes f_0, u \in \sdomain {f_0 \otimes f_0}
\end{equation*}
and there exists a neighborhood $I$ of $[0,1]$ where the one dimensional exponential family

\begin{equation*}
  f_t = \euler^{tu - K_0(tu)} \cdot f_0 \otimes f_0, \quad t \in I
\end{equation*}
exists. To show $\expectat {f_0 \otimes f_0}{(f/f_0\otimes f_0)^t} < + \infty$ for $t \in I$ it is enough to consider the convex cases $t < 0$ and $t> 1$. We have

\begin{equation*}
  \integrald {S_2} {f \circ A_x} {\sigma(dx)} = \integrald {S_2} {\euler^{fu\circ A_x - K_0(u)}} {\sigma(dx)} \cdot f_0 \otimes f_0
\end{equation*}
and in the convex cases

\begin{multline*}
  \expectat {f_0 \otimes f_0} {\left(\frac{\integrald {S_2} {f \circ A_x} {\sigma(dx)}}{f_0 \otimes f_0}\right)^t} = \expectat {f_0 \otimes f_0} {\left(\integrald {S_2} {\euler^{u \circ A_x - K_0(u)}} {\sigma(dx)}\right)^t} \le \\ \expectat {f_0 \otimes f_0} {\integrald {S_2} {\euler^{tu \circ A_x - tK_0(u)}} {\sigma(dx)}} = \expectat {f_0 \otimes f_0} {\euler^{tu  - tK_0(u)}} = \euler^{K_0(tu) - t K_0(u)}
\end{multline*}

The last equation if Bayes' formula for conditional expectation.
\end{proof}

\begin{definition}
For each element of the maximal exponential model containing $f_0$, $f \in\maxexp{f_0}$, the Boltzmann operator is

  \begin{multline*}
Q(f)(v) = \\
\int_{\reals^3}\int_{S^2} (f(v - x {x'}(v - w))f(w + x {x'}(v - w))-f(v)f(w))\absoluteval{{x'}(v - w)}\ \sigma(dx)\ dw,    
\end{multline*}
\end{definition}

In our definition we have restricted the domain of the Boltzmann operator to a maximal exponential model containing the standard normal density in order to fit into our framework and be able to prove the smoothness of the operator. The maximal exponential model $\maxexp {f_0}$ contains all normal densities $f \sim \text{N}(\mu,\Sigma)$. It has other peculiar properties.

As $f \in \maxexp{f_0}$, $f = \euler^{u - K_0(u)} \cdot f_0$, $u$ belongs to the interior of the proper domain of $K_0$, $u \in \mathcal S_{f_0} \subset \eBspace{f_0}$. It follows from Prop. \ref{prop:maxexp-pormanteaux} that we have the equality and isomorphism of the Banach spaces $L^{\Phi}(f)$ and $L^{\Phi}(f_0)$. For the random variable $V_a \colon v \mapsto \absoluteval v ^a$ it holds $V_\alpha  \in L^{\Phi}(f_0) = L^{\Phi}(f)$ for all $a \in [1,2]$. In fact,

\begin{equation*}
  \expectat {f_0} {\cosh(\alpha V_a)} = (2\pi)^{-3/2} \int_{\reals^3} \cosh(\alpha\absoluteval v ^a) \expof{-\absoluteval v^2/2}\ dv
\end{equation*}
is finite for all $\alpha$ if $a \in [0,2[$ and for $\alpha < 1/2$ if $a = 2$. In particular, it follows that $V_1(v) =\absoluteval v$ has finite moments with respect to $f$, $\int \absoluteval v^n  f(v) \ dv < + \infty$, $n = 1,2,\dots$. 
%
%

As $x'(v - w) = -x'(v_x - w_x)$, the measure $\absoluteval{{x'}(v - w)}\ dvdw$ is invariant under the transformation $A_x$ and the measure $f(v_x)f(w_x)\absoluteval{x'(v-w)}\ dvdw$ is the image of $f(v)f(w)\absoluteval{x'(v-w)}\ dvdw$ under $A_x$. Other properties are obtained in the proof of the following proposition.

\begin{proposition}\label{prop:Boltzmann-field}
Let $f_0(v) = (2\pi)^{-3/2}\expof{-\absoluteval v ^2 / 2}$ and $f \in \maxexp{f_0}$. Then $Q(f)/f \in \mBspace f$. Then $f \mapsto Q(f)/f$ is a vector field in the pretangent bundle $\prescript{*}{}T\maxexp {f_0}$ called \emph{Boltzmann field}.
\end{proposition}
\begin{proof}
Let us consider first the second part of the Boltzmann operator

\begin{equation*}
  Q_-(f)(v) = \int_{\reals^3}\int_{S^2} f(v)f(w)\absoluteval{{x'}(v - w)}\ \sigma(dx)\ dw
= f(v) \int_{\reals^3}f(w) \left(\int_{S^2} \absoluteval{{x'}(v - w)}\ \sigma(dx)\right)\ dw.
\end{equation*}

Note that from inequality \eqref{eq:delta2}

\begin{equation*}
  \Phi_*\left(\int_{S^2} \absoluteval{{x'}(v - w)}\ \sigma(dx)\right) = \Phi_*\left(\absoluteval{v-w} \int_{S^2} \absoluteval{x_1} \ \sigma(dx)\right) \le \left(\int_{S^2} \absoluteval{x_1} \ \sigma(dx)\right)^2 \Phi_*\left(\absoluteval{v-w}\right).
\end{equation*}

We prove $Q_-(f)/f \in L^{\Phi_*}(f)$:

\begin{align*}
  \expectat f {\Phi_*\left(\frac{Q_-(f)}f\right)} &= \int_{\reals^3} \ dv f(v) \Phi_*\left(\int_{\reals^3}f(w) \left(\int_{S^2} \absoluteval{{x'}(v - w)}\ \sigma(dx)\right)\ dw\right) \\ &\le
\int_{\reals^3} \ dv f(v) \int_{\reals^3}\ dw f(w) \Phi_*\left(\int_{S^2} \absoluteval{{x'}(v - w)}\ \sigma(dx)\right) \\ &= \int_{\reals^3} \ dv f(v) \int_{\reals^3}\ dw f(w) \Phi_*\left(b \absoluteval{v-w}\right) \\ &\le \frac{b^2}2 \int_{\reals^3}\int_{\reals^3}\ dvdw\  f(v)f(w) \absoluteval{v-w}^2,
\end{align*}
which is finite as $\absoluteval{v-w}^2 \le 2(\absoluteval u^2 + \absoluteval v^2)$.

We consider now the first part of the Boltzmann operator

\begin{align*}
  Q_+(f)(v) &= 
\integrald{\reals^3}{\integrald{S^2} {f(v - x {x'}(v - w))f(w + x {x'}(v - w))\absoluteval{{x'}(v - w)}}{\sigma(dx)}}{dw}
\\ &= 
\integrald{\reals^3}{\integrald{S^2} {f(v_x)f(w_x)\absoluteval{{x'}(v - w)}}{\sigma(dx)}}{dw}
\end{align*}
We want to prove that $Q_+(f)/f \in L^{\Phi_*}(f)$ or, equivalently, $Q_+(f)/f_0 \in L^{\Phi_*}(f_0)$. As $f \in \maxexp {f_0}$, we can write $f$ as $f = \euler^{u - K_0(u)} \cdot {f_0}$, where $u \in \eBspace{f_0}$, so that

\begin{equation*}
  Q_+(f)(w) = f_0(w) \int_{S^2} \sigma(dx) \int_{\reals^3} dv f_0(v) \euler^{u(v_x)+u(w_x)-2K_0(u)} \absoluteval{{x'}(v - w)}
\end{equation*}
and
 
\begin{align*}
\Phi_*\left(\frac{Q_+(f)(w)}{f_0(w)}\right) &\le \int_{S^2} \sigma(dx) \int_{\reals^3} dv f_0(v) \Phi_*\left(\euler^{u(v_x)+u(w_x)-2K_0(u)} \absoluteval{{x'}(v - w)}\right) \\ &\le \int_{S^2} \sigma(dx) \int_{\reals^3} dv f_0(v) L(\absoluteval{{x'}(v - w)})\Phi_*\left(\euler^{u(v_x)+u(w_x)-2K_0(u)}\right)
\end{align*}
where $L(a) = a \vee a^2$. It follows

\begin{multline*}
\Phi_*\left(\frac{Q_+(f)(w)}{f_0(w)}\right) \\ \le \int_{S^2} \sigma(dx) \int_{\reals^3} dv f_0(v) L(\absoluteval{{x'}(v - w)})\left((u(v_x)+u(w_x)-2K_0(u))\euler^{u(v_x)+u(w_x)-2K_0(u)}+1\right)
\end{multline*}
and

\begin{multline*}
\expectat {f_0} {\Phi_*\left(\frac{Q_+(f)}{f_0}\right)} 
\\ \le \iint_{\reals^3} dvdw f_0(v) f_0(w)L(\absoluteval{{x'}(v - w)})\left((u(v)+u(w)-2K_0(u))\euler^{u(v)+u(w)-2K_0(u)}+1\right)
\\ \le \iint_{\reals^3} dvdw f(v) f(w)L(\absoluteval{{x'}(v - w)})\left(u(v)+u(w)-2K_0(u)\right) \\ + \iint_{\reals^3} dvdw f_0(v) f_0(w)L(\absoluteval{{x'}(v - w)})
\end{multline*}
where both terme are finite.

Finally, the integral of the Boltzmann operator is zero:

  \begin{multline*}
\int_{\reals^3} Q(f)(v)\ dv= \\ \int_{S^2} \int_{\reals^3} \int_{\reals^3}(f(v_x)f(w_x)-f(v)f(w))\absoluteval{{x'}(v - w)}\ dw\ dv\ dx = \\
\int_{S^2} \int_{\reals^3} \int_{\reals^3}f(v_x)f(w_x)\absoluteval{{x'}(v_x - w_x)}\ dw_x\ dv_x\ dx - \\ \int_{S^2} \int_{\reals^3} \int_{\reals^3}f(v)f(w))\absoluteval{{x'}(v - w)}\ dw\ dv\ dx = 0.   
  \end{multline*}
\end{proof}

The smoothness of the Boltzmann field could be studied by carefully analyzing the structure of the operator as superposition of

\begin{enumerate}
\item Product: $\maxexp {f_0} \ni f \mapsto f\otimes f \in \maxexp{f_0\otimes f_0}$;
\item Interaction: $\maxexp{f_0\otimes f_0} \ni f\otimes f \mapsto g = B f\otimes f \in \maxexp{f_0\otimes f_0}$;
\item Conditioning: $\maxexp{f_0\otimes f_0} \ni g \mapsto \integrald {S_2}{g\circ A_x}{\sigma(dx)} \in \maxexp{f_0\otimes f_0}$;
\item Marginalization.
\end{enumerate}

The single operations of the chain are discussed in \cite{pistone|rogantin:99}. We do not do this analysis here, and conclude the section by rephrasing in our language the Maxwell form of the Boltzmann operator.

\begin{proposition}\label{prop:maxwell}
  Let $f \in \maxexp{f_0}$ and $g \in L^{\Phi}(f)$. Then $Ag$ defined by
 
 \begin{equation*}
    Ag(v,w) = \integrald{S^2}{\frac12(g(v_x)+g(w_x))}{\sigma(dx)} - \frac12(g(v)+g(w))
  \end{equation*}
belongs to $L^{\Phi}(f\otimes f)$ and

  \begin{equation*}
    \scalarat f g {Q(f)/f} = \expectat {f\otimes f} {Ag}.
  \end{equation*}
Especially, if $f = \euler^{u - K_0(u)} \cdot f_0$

  \begin{equation*}
    \scalarat f u {Q(f)/f} = \expectat {f\otimes f} {A\left(\frac{f}{f_0}\right)}.
  \end{equation*}
\end{proposition}

\section{Conclusion and Discussion}
\label{sec:conclusion}

We have shown that a careful consideration of the relevant functional analysis allows to discuss some basic features of statistical models of interest in Statistical Physics in the framework of the nonparametric Information Geometry based on Orlicz spaces. In particular, we have defined the exponential statistical manifold and its vector bundles, namely the tangent bundle, the pretangent bundle, the Hilbert bundle. Partial results are obtained on connections, which is a topic considered by many Authors the very core of statistical manifolds theory. 

For example, the Boltzmann equation takes the form of an evolution equation for the Boltzmann field

\begin{equation*}
  \delta f_t = \frac{Q(f_t)}{f_t}, \quad \delta f \in T\maxexp{f_0}, \frac{Q(f)}{f} \in \prescript{*}{}T\maxexp{f_0},
\end{equation*}
and we can compute the covariant derivative of Boltzmann entropy along the Boltzmann field $D_{Q(f)/f} E(f) = \scalarat f {Q(f)/f}{\lnof f - E(f)}$ with Prop. \ref{prop:maxwell}, cfr. \cite[Ch. 3]{villani:2002review}. Our treatment of Boltzmann entropy and Boltzmann equation doen not add any new result, but our aim is to transform a generic geometric intuition about the geometry of probability densities into a formal geometrical methodology.

A number of issues remain open, in particular the proper topological setting of the second order structures and the proper definition of sub-manifold, an important topic which is not mentioned at all in this paper. 

In the case of the pretangent bundle, we have been able to show that it is actually the tangen bundle of an extension of the exponential manifold, the mixture manifold, $\prescript{*}{}T\pdensities \to T\sdensities$. It has been the object of much research the construction of a manifold whose tangent space would be the Hilbert bundle. In some sense the answer is known because of the embedding $p \mapsto \sqrt p$ that maps positive densities $\pdensities$ into the unit sphere $S_\mu$, but a proper definition of the charts is difficult in this setting. 

It has been suggested to use functions called deformed exponentials to mimic the theory of exponential families, see the monograph \cite{naudts:2011GTh}, and also \cite{pistone:2009EPJB}, \cite[Sec. 5]{pistone:2013GSI2013}. An example of deformed exponential is

\begin{equation*}
  \exp_d (u) = \left(\frac 12 u + \sqrt{1 + \frac14 u^2} \right)^2
\end{equation*}
which is a special case of the class introduced in \cite{kaniadakis:2002PhRE,kaniadakis:2005PhRE}. 

The function $\exp_d$ maps $\reals$ onto $\reals_>$, is increasing, convex, and 

\begin{equation*}
  \Phi_d(u) = \frac 12 (\exp_d(u) + \exp_d(-u) ) - 1 = \frac12 u^2.
\end{equation*}

The Young conjugate is $\Phi_{d,*} = \Phi_d$ and the Orlicz space is $L^{\Phi_d}(p) = L^2(p)$
A nonparametric exponential family around the positive density $p$ was defined by \cite{vigelis|cavalcante:2011} to be

\begin{equation*}
  q = \exp_d \left(u - K_p(u) + \ln_d p\right),
\end{equation*}
where

\begin{equation*}
  \ln_d(v) = \exp_d^{-1}(v) = v^{1/2} - v^{-1/2}.
\end{equation*}
If we assume $\expectat {\bar p} u = 0$, where $\bar p$ is a suitable density associate to $p$, then 

\begin{equation*}
  K_p(u) = \expectat {\bar p} {\ln_d p - \ln_d q}.
\end{equation*}

An account of this research in progress will be publisher elsewhere. We conclude by mentioning the different nonparametric approach of \cite{ay|jost|le|schwachhofer:2013arXiv1207.6736}.
 

\section*{Acknowledgements}

This research was supported by the de Castro Statistics Initiative, Collegio Carlo Alberto, Moncalieri. I wish to thank the Guest Editor Antonio Scafone for suggesting me to present this contribution and Bertrand Lods for helpful conversations of the Boltzmann equation.

\bibliographystyle{mdpi}
\makeatletter
\renewcommand\@biblabel[1]{#1. }
\makeatother

\end{document}